\documentclass[reqno]{amsart}
\usepackage{amssymb,latexsym}
\usepackage{amsmath}
\usepackage{hyperref}
\usepackage{url}
\usepackage{graphicx}

\numberwithin{equation}{section}
\theoremstyle{plain}
 \newtheorem{theorem}{Theorem}[section]
 \newtheorem{lemma}[theorem]{Lemma}
 \newtheorem{proposition}[theorem]{Proposition}
 \newtheorem{corollary}[theorem]{Corollary}

\newcommand \figwidthcoeff  {0.95} 
\newcommand \quotat[1] {``#1''}
\newcommand \Part [1] {\textup{PLat}(#1)}
\newcommand \Pfin [1] {\textup{Pfin}(#1)}
\newcommand \Equ [1] {\textup{Equ}(#1)}
\newcommand{\Nplu}{\mathbb{N}^+}
\newcommand{\Nnul}{\mathbb{N}_0}
\newcommand \set[1] {\{#1\}}
\newcommand \prt {\textup{prt}}
\newcommand \topA {\mathbf 1_{\Part A}}
\newcommand \botA {\mathbf 0_{\Part A}}
\newcommand \stskip{\bigskip} 
\newcommand \onlyext[1]{#1}

\newcommand\semmi[1]{}

\begin{document}

\title[Four-element generating sets of partition lattices]
{Atoms in four-element generating sets of partition lattices}

\author[G.\ Cz\'edli]{G\'abor Cz\'edli}
\email{czedli@math.u-szeged.hu}
\urladdr{http://www.math.u-szeged.hu/~czedli/}
\address{University of Szeged, Bolyai Institute. 
Szeged, Aradi v\'ertan\'uk tere 1, HUNGARY 6720}



\date{\hfill \textbf{October 25, 2024}}

\dedicatory{This paper is dedicated to \'Eva, Gerg\H{o}, M\'arti, Andi, and M\'at\'e}

\begin{abstract}
Since Henrik Strietz's 1975 paper proving that the lattice Part($n$) of all partitions of an $n$-element finite set is four-generated,  more than half a dozen papers have been devoted to four-element generating sets of  this lattice.
We prove that each element of  Part($n$) with height one or two (in particular, each atom) belongs to a four-element generating set. Furthermore, our construction leads to a concise and easy proof of a 1996 result of the author stating that the lattice of partitions of a countably infinite set is four-generated 
as a complete lattice.
In a recent paper \quotat{Generating Boolean lattices by few elements and exchanging session keys}, see \href{https://doi.org/10.30755/NSJOM.16637}{https://doi.org/10.30755/NSJOM.16637}, the author establishes a connection between cryptography and small generating sets of some lattices, including Part($n$). Hence, it is worth pointing out that by combining a construction given here 
with a recent paper by the author, 
\quotat{\emph{Four-element generating sets with block count width at most two in partition lattices}}, available at \href{https://tinyurl.com/czg-4gw2}{https://tinyurl.com/czg-4gw2}, we obtain many four-element generating sets of Part($n$).
\end{abstract}

\maketitle

\section{Introduction}
This paper belongs to lattice theory. Apart from understanding the definition of a sublattice of a lattice as an algebraic system, the paper requires no prior knowledge from the reader.

For a set $A$, the partitions of $A$ form a complete lattice, the \emph{\textbf partition \textbf{lat}tice} $\Part A$ of $A$. This lattice is isomorphic to the \emph{equivalence lattice} $\Equ A$ of $A$; in fact, it is the canonical bijective correspondence between the partitions and the equivalences of $A$ that defines the lattice order and so the lattice structure of $\Part A$: For partitions $\alpha,\beta\in\Part A$, $\alpha\leq \beta$ means that every pair in the equivalence relation determined by $\alpha$ belongs to the equivalence relation determined by $\beta$. As usual, a subset $X$ of $\Part A$ is a \emph{generating set} of $\Part A$ if no proper sublattice of $\Part A$ includes $X$ as a subset. Similarly, if no proper complete sublattice of $\Part A$ includes $X$, then we say that \emph{$X$ generates $\Part A$ as a complete lattice}.
For $k\in\Nplu:=\set{1,2,3,\dots}$, a lattice $L$ is $k$-generated if it has a $k$-element generating set. For $n\in\Nnul:=\set0\cup\Nplu$, we denote $\set{i\in\Nplu: i\leq n}$ by $[n]$, and we write $\Part n$ rather than $\Part{[n]}$.
By Strietz \cite{strietz75} and \cite{strietz77}, $\Part n$ is four-generated but not three-generated provided that $4\leq n\in\Nplu$. 
Since his papers, more than half a dozen papers have been devoted to the four-element generating sets of partitions lattices; see the historical mini-survey subsection in \cite{czgsesskey}.  See also the other papers in the bibliographic section of the present paper, in particular, see \cite{czgirk},  \cite{czgerl24}, and \cite{czgpairhor} for the most recent developments.


To present a part of our motivation, assume that  $\vec x:=(x_1,\dots,x_k)\in\Part n^k$ 
such that $k\in \Nplu$ is small, $\set{x_1,\dots,x_k}$ is a generating set of $\Part n$, and $\vec p=(p_1,\dots, p_b)$ is a vector of $k$-ary lattice terms.    Roughly saying, \cite[Proposition 5.1]{czgsesskey} implies that computing $\vec x$ from $\vec p$ and $\vec p(\vec x)$ is an NP-hard problem. Hence, hopefully, 
if $A$ and $B$ are two communicating parties who have previously agreed upon a secret key $\vec x$,  then they can change $\vec p$ on an open channel from time to time and use $\vec p(\vec x)$ as a \emph{session key} in a secret-key cryptosystem. 
Note at this point that  $\vec x$ in itself cannot be a (permanent) secret key; otherwise, the adversary could uncover $\vec x$ when he guesses the content of a, say, Vernam-cipher-encrypted message, and he could decrypt all further messages. 
A complete section in \cite{czgsesskey} warns the reader that no rigorous theoretical treatment approves this idea concerning modern cryptographic criteria. 
As these criteria are neither met by some popular cryptosystems like RSA, the idea given in \cite{czgsesskey} still has some motivating value for lattice theory and leads to the following conclusion: If we could construct \emph{very many} four-element generating sets of $\Part n$, then a random choice out of these constructible sets (augmented with a few further random partitions) \emph{might function} as a secret key. This gives some justification to our effort to find four-element generating sets of $\Part n$.

In addition to the paragraph above, there are also strictly lattice theoretical motivations. First, partition lattices play a central role in lattice theory, since they have nice properties and, say, congruence lattices are naturally embedded in partition lattices. Second, there are several earlier results on four-element generating sets of partition lattices, where the generating sets possess specific properties.
The first such property is that the set in question has two comparable members; in chronological order, Strietz \cite{strietz75}--\cite{strietz77}, Z\'adori \cite{zadori}, \cite{CzGEateq}, \cite{czgeek},  and \cite{czgoluoch} contain results on four-element generating sets with this property.  Some other properties are considered in 
\cite{czgpairhor},  \cite{czgirk}, and \cite{czgerl24}. Sometimes, different approaches to four-element generating sets can be combined, and this leads to many new four-element generating sets; the present paper exemplifies this by (the proof of) Corollary \ref{cor:nm}.

In $\Part A$, the \emph{height} of an atom is $1$; a partition $\alpha\in\Part A$ is an \emph{atom} if it has a two-element block and the rest of its blocks are singletons. If $\alpha\in\Part A$ is the join of two distinct atoms, then $\alpha$ is said to be of \emph{height $2$}. There are two sorts of partitions with height 2. Namely, $\alpha\in\Part A$ is of height 2 if and only if either $\alpha$ has two two-element blocks and the rest of its blocks are singletons or $\alpha$ has a three-element block and the rest of its blocks are singletons. 
According to these two possibilities, we say that $\alpha$ is of \emph{type} $2+2$ or it is of \emph{type} $3$, respectively. The paper deals with the following three properties of a four-element generating set $X$ of $\Part n$: $X$ contains an atom, $X$ contains a partition of height $2$ and type $2+2$, and $X$ contains a partition of height $2$ and type $3$.   However, it is reasonable to formulate the main result, Theorem \ref{thm:main}, more concisely.

Implicitly, the historical comment in \cite{czgerl24}, a recent paper,  may suggest that an atom in a four-element generating set of $\Part n$ for a large $n$ is probably impossible. Now it turns out that it is possible, and \cite{czgerl24} was the direct predecessor that inspired the present paper.

\section{Results}
For $n\leq 3$, $|\Part n|\leq 5$ and nothing interesting can be stated.
 Hence, the main result below assumes that $n\geq 4$. The only element of $\Part n$ with height $0$ is its smallest element $\mathbf 0_{\Part n}$.  
From Strietz's previously mentioned result asserting that $\Part n$ is not three-generated, it is not hard to see that the following theorem does not hold for 
$\alpha=\mathbf 0_{\Part n}$.

\stskip

\begin{theorem}\label{thm:main}
Assume that $4\leq n\in\Nplu$, and let $\alpha\in\Part n$ be a partition of height $1$ or $2$. Then there exist $\beta,\gamma,\delta\in\Part n$ such that $\set{\alpha,\beta,\gamma,\delta}$ is a four-element generating set of $\Part n$.
\end{theorem}

\stskip

Let $\Pfin A$ denote the sublattice of $\Part A$ consisting of all partitions that have only finitely many non-singleton blocks and each of these blocks is finite.
For a subset $\Phi$ of $\Part A$, we denote by $[\Phi]$ the sublattice generated by $\Phi$. We will present a new concise proof of the following result.

\stskip

\begin{proposition}[\cite{CzGEateq}]\label{prop:aleph0} 
Let $A$ be a countably infinite set, and let $\alpha$ be an atom of $\Part A$. Then there are $\beta,\gamma,\delta\in\Part A$ such that $\Phi:=\set{\alpha,\beta,\gamma,\delta}$ generates $\Part A$ \emph{as a complete lattice} and, furthermore, $\Pfin A$ is included in the \emph{sublattice} generated by $\Phi$, that is, $\Pfin A\subseteq [\Phi]$.
\end{proposition}

\stskip
In the proposition above, $[\Phi]$ and $\Part A$ are of cardinalities $\aleph_0$ and $2^{\aleph_0}$, whereby  $[\Phi]\neq\Part A$; furthermore,  $[\Phi]$ is not a complete sublattice of $\Part A$.

For $\alpha\in\Part A$ and $u\in A$, $u/\alpha$ stands for the $\alpha$-block of $u$, that is, for the unique block of $\alpha$ that contains $u$. 
For sets $A\subseteq B$ and partitions $\alpha\in\Part A$ and $\beta\in\Part B$, we say that $\beta$  \emph{extends}  $\alpha$ if for every $u\in A$, $u/\alpha=A\cap(u/\beta)$. 
Note that for $n,m\in\Nplu$,  $[n]\subseteq [n+m]$. 
We will derive the following statement from
a construction needed in the proof of 
 Theorem \ref{thm:main} and the proof of \cite[Theorem 1]{czgirk}.

\stskip

\begin{corollary}\label{cor:nm}
Assume that $m\in\Nplu$ is an even number,    $5\leq n\in\Nplu$ is an odd number, and $\alpha$ is an atom in $\Part n$. Then there are $\beta,\gamma,\delta\in\Part n$ such that the following two facts hold.
\begin{enumerate}
\item\label{cornma}$\set{\alpha,\beta,\gamma,\delta}$ is a four-element generating set of $\Part n$.
\item\label{cornmb}  $\Part{n+m}$ has at least 
$2^{m-3}\cdot(m-1)! / (3m+3)$  four-element generating sets $\set{\alpha',\beta',\gamma',\delta'}$ such that $\alpha'$, $\beta'$, $\gamma'$, and $\delta'$ extend $\alpha$, $\beta$, $\gamma$, and $\delta$, respectively.
\end{enumerate}  
\end{corollary} 

\stskip

\section{Proofs}
A partition with \emph{nonsingleton} blocks $\set{a_{1,1},\dots,a_{1,t_1}}$, \dots, 
$\set{a_{s,1},\dots,a_{s,t_s}}$ will be denoted by 
\[\prt(a_{1,1}\dots a_{1,t_1};\dots; a_{s,1}\dots a_{s,t_s}) \text{ or, if confusion threatens,  }
\prt(\set{a_{1,1},\dots, a_{1,t_1}};\dots; \set{a_{s,1},\dots, a_{s,t_s}}).
\]
For example, $\prt(23)$ and $\prt(13;24)$ are members of $\Part 7$; the former is an atom, the latter is of height $2$. However, we do not drop the commas and the curly brackets when 
dealing with $\Part n=\Part{[n]}$ for $n\geq 10$ or an unspecified $n$, since otherwise, say, $12$ (twelve) and the list $1,2$ could be confused. 
Note the \quotat{commutativity} of $\prt$; e.g.,  $\prt(x y)=\prt(y x)$.
In some form, the following trivial lemma occurs in many earlier papers; see, e.g., \cite[Lemma 2.5]{czgoluoch}.

\stskip

\begin{lemma}\label{lemma:circle} 
Assume that $3\leq k\in\Nplu$. Let $\set{a_1,\dots,a_k}$ be a $k$-element subset of a set $A$, and denote by $S$ the sublattice generated by $Y:=\set{\prt(a_i a_{i+1}): i\in[k-1]}\cup\set{\prt(a_k a_1)}$ in $\Part A$. Then for all $i,j\in[k]$ such that $i\neq j$, the partition $\prt(a_i a_j)$ belongs to $S$. Consequently, if $|A|=k$, then $Y$ is a generating set of $\Part A$. 
\end{lemma}

\stskip

The following four lemmas constitute the lion's share of the proofs of Theorem \ref{thm:main} and Corollary \ref{cor:nm}. Note that the corresponding constructions are visualized by Figure \ref{figa}, which adheres to the following convention. The non-singleton $\alpha$-blocks are denoted by ovals. 
Each of $\kappa\in\set{\beta,\gamma,\delta}$ has its own \emph{line style}, consisting of a color, a thickness,  and a feature (solid, dotted, dashed), which we use for the so-called \emph{$\kappa$-edges} of our graphs.   For $x,y\in A$, $x$ and $y$ belong to the same $\kappa$-block if and only if they can be connected by a path consisting of $\kappa$-edges of the graph. (The length of this path can be 0, allowing $x=y$.)
If $\kappa$ and $\kappa_0$ both occur in our arguments, then only a textual explanation specifies which $\kappa$-colored edges define $\kappa_0$; there 
are no separate  $\kappa_0$-styled edges.

\stskip

\begin{lemma}\label{lemma:oddat}
For $2\leq k\in\Nplu$ and $n=2k+1$, let $A=\{a_1,\dots, a_k, b_1,\dots,b_k, a_{k+1}=b_{k+1}\}$ be an $n$-element set. Let  $\alpha:=\prt(a_1 b_1)$,
\begin{align}
\beta_0&:=\prt(a_1 b_2;a_2 b_3;\dots;a_{k-1} b_{k})=\bigvee_{i\in[k-1]}\prt(a_i b_{i+1}), & \beta&:=\beta_0\vee \prt(a_k a_{k+1}), & 
\label{eq:loddata} 
\\
\gamma_0&:=\prt(b_1 a_2;b_2 a_3;\dots;b_{k-1} a_{k})=\bigvee_{i\in[k-1]}\prt(b_i a_{i+1}),  &\gamma&:=\gamma_0\vee \prt(b_k b_{k+1}), &
\label{eq:loddatb}
\end{align}
and $\delta :=\prt(a_1\dots a_k; b_1\dots b_k)$. Denote by $S_0$ and $S$ the sublattices generated by $\Phi_0:=\{\alpha,\beta_0, \gamma_0, \delta\}$ and $\Phi:=\set{\alpha,\beta,\gamma,\delta}$, respectively, in $\Part A$. Then 
\begin{enumerate}
\item[\textup{(a)}]\label{lemma:oddata}
for all $x,y\in A\setminus\set{a_{k+1}}=A\setminus\set{b_{k+1}}$ such that $x\neq y$, $\prt(x y)$ belongs to $S_0$, and
\item[\textup{(b)}]\label{lemma:oddatb}
$S=\Part A$, that is, $\Phi$ generates $\Part A$.
\end{enumerate}
\end{lemma}

\begin{figure}[ht] 
\centerline{ \includegraphics[width=\figwidthcoeff\textwidth]{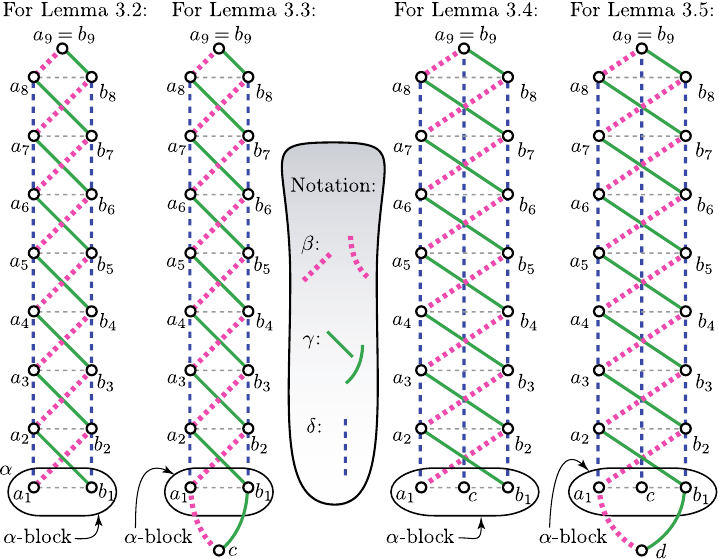}} \caption{The constructions for Lemmas \ref{lemma:oddat}--\ref{lemma:oddhtwo}, with $k:=8$}\label{figa}
\end{figure}

\stskip

\begin{proof} For $k=8$, our partitions are visualized by the graph on the left of Figure \ref{figa}.
In addition to the edges defining $\beta$, $\gamma$, and $\delta$, the graph also contains the
horizontal grey dashed edges $(a_i,b_i)$ for $i\in [k]$; the same applies for the second graph in the figure.  For the third and fourth graphs, the  $(a_i,b_i)$s are edges only for $2\leq i\leq k$, but the grey-dashed
$(a_1,c)$ and $(c,b_1)$ are  edges, too. 

Observe that $\beta_0=\beta\wedge (\alpha\vee\delta)\in S$ and $ \gamma_0=\gamma\wedge (\alpha\vee\delta)\in S$. Hence, $S_0\subseteq S$. 
It suffices to show that 
\begin{align}
&\text{for every edge }(x,y) \text{ of the graph such that  } x\neq a_{k+1}\neq y, \text{ }\, \prt(x y)\in S_0\text{, and}
\label{eq:mndL9gra}
\\
&\text{for every edge }(x,y)\text{ of the graph, we have that }\prt(x y)\in S.
\label{eq:mndL9grb}
\end{align}
Indeed, if \eqref{eq:mndL9gra} and \eqref{eq:mndL9grb} hold, then Lemma \ref{lemma:circle} applied to the  the \quotat{perimeter} of the subgraph $A\setminus\set{a_{k+1}=b_{k+1}}$  and to the 
\quotat{perimeter} of the whole graph yields Parts (a) and (b) of Lemma \ref{lemma:oddat}, respectively. Let us compute; note that $(a_i,b_i,\beta,\gamma)$ and  $(b_i,a_i,\gamma,\beta)$  play the same role by symmetry, so a part of our computations need no separate checking.
\allowdisplaybreaks{
\begin{align}
\prt(a_1 b_1)={}&\alpha\in S_0,  \label{align:odda}\\ 
\epsilon:={}&\bigvee_{i=1}^{k-2} \prt(a_i a_{i+2}; b_i b_{i+2})=(\beta_0\vee \gamma_0)\wedge \delta \in S_0.  \label{align:oddb}
\end{align}
}
We can assume that the graph is drawn so that the geometric distance of  $a_i$ and $a_{i+1}$ and that of  $b_i$ and $b_{i+1}$  are 1 for $i\in[k-1]$. Then
for $x\neq y\in A$, $x$ and $y$ belong to the same $\epsilon$-block if and only if they lie on the same vertical geometric line and their distance is an even integer. This visual idea helps to understand the rest of the computations in some places.  Along the graph, we  proceed upwards; 
each containment \quotat{$\in S_0$}  below follows from the preceding containments within the list \eqref{align:odda}--\eqref{align:oddtket}, 
$\Phi_0\subseteq S_0$,  and the commutativity of our notation.
\allowdisplaybreaks{
\begin{gather}
\prt(a_1 a_2)=(\alpha\vee \gamma_0)\wedge\delta \in S_0,  
\quad
\prt(b_1 b_2)=(\alpha\vee \beta_0)\wedge\delta \in S_0,
\label{align:oddot}
\\
\prt(a_1 b_2)=\bigl(\prt(a_1 b_1)\vee \prt(b_1 b_2)\bigr)\wedge\beta_0\in S_0,
\quad
\prt(b_1 a_2)=\bigl(\prt(b_1 a_1)\vee \prt(a_1 a_2)\bigr)\wedge\gamma_0\in S_0,
\label{align:oddhat}
\\
\prt(a_2 b_2)=\bigl(\prt(a_2 a_1)\vee\prt(a_1 b_2\bigr) \wedge \bigl(\prt(a_2 b_1)\vee\prt(b_1 b_2\bigr)\in S_0,
\label{align:oddhet}
\\
\prt(a_2 a_3; b_1 b_2)=\bigl(\prt(a_2 b_2)\vee \gamma_0\bigr)\wedge\delta\in S_0,
\quad
\prt(b_2 b_3; a_1 a_2)=\bigl(\prt(b_2 a_2)\vee \beta_0\bigr)\wedge\delta\in S_0,
\label{align:oddnyolc}
\\
\left.\begin{aligned}
\prt(a_1 a_3)=\bigl(\prt(a_2 a_3; b_1 b_2)\vee \prt(a_1 a_2)\bigr)\wedge\epsilon\in S_0,
\cr
\prt(b_1 b_3)=\bigl(\prt(b_2 b_3; a_1 a_2)\vee \prt(b_1 b_2)\bigr)\wedge\epsilon\in S_0,
\end{aligned} \,\,\right\}
\label{align:oddkil}\\
\left.\begin{aligned}
\prt(a_2 a_3)=\bigl(\prt(a_2 a_1)\vee\prt(a_1 a_3)) \wedge \prt(a_2 a_3; b_1 b_2)\in S_0,  \cr
\prt(b_2 b_3)=\bigl(\prt(b_2 b_1)\vee\prt(b_1 b_3)) \wedge \prt(b_2 b_3; a_1 a_2)\in S_0,
\end{aligned} \,\,\right\}
\label{align:oddtiz}
\\
\left.\begin{aligned}
\prt(a_3 b_2)=\bigl(\prt(a_3 a_2)\vee \prt(a_2 b_2)\bigr)\wedge \gamma_0\in S_0,\cr
\prt(b_3 a_2)=\bigl(\prt(b_3 b_2)\vee \prt(b_2 a_2)\bigr)\wedge \beta_0\in S_0,
\end{aligned} \,\,\right\}
\label{align:oddtegy}
\\
\prt(a_3 b_3)=\bigl(\prt(a_3 b_2)\vee\prt(b_2 b_3)\bigr) \wedge \bigl(\prt(a_3 a_2)\vee\prt(a_2 b_3)\bigr) \in S_0.
\label{align:oddtket}
\end{gather}
}
We have seen so far that  $\prt(a_3 b_3)\in S_0$ and for every edge $(x,y)$ of the graph that is (geometrically) below $(a_3,b_3)$, the partition $\prt(x y)$ is in $S_0$. 
(Of course, we stop here if $k=3$, and  we stop right after \eqref{align:oddhet} if $k=2$.) As an induction hypothesis, assume  that $3\leq i\in [k-1]$ and
\begin{equation}
\prt(a_i b_i)\in S_0\text{ and  }\prt(x y)\in S_0\text{ for every edge }(x,y)\text{  below }(a_i,b_i).
\label{eq:mWghpstx}
\end{equation} 
Repeating \eqref{align:oddnyolc}--\eqref{align:oddtket} so that we change the subscripts $1$, $2$, and $3$ to $i-1$, $i$, $i+1$, respectively, we obtain the validity of \eqref{eq:mWghpstx} for $i+1$.  Therefore, it follows by induction that 
\eqref{eq:mWghpstx} holds for $i=k$. Hence, we have proved \eqref{eq:mndL9gra}.
Finally, adding  
\begin{equation}
\prt(a_k a_{k+1})=\bigl(\prt(a_k b_k)\vee \gamma \bigr)\wedge \beta\in S\,\text{ and }\,
\prt(b_k b_{k+1})=\bigl(\prt(a_k b_k)\vee \beta \bigr)\wedge \gamma\in S 
\label{eq:VjrW6gh}
\end{equation}
to \eqref{eq:mndL9gra}, we obtain that \eqref{eq:mndL9grb} also holds,
completing the proof of Lemma \ref{lemma:oddat}.
\end{proof}

\stskip

\begin{lemma}\label{lemma:evenat}
For $3\leq k\in\Nplu$ and $n=2k+2$, let $B=\{a_1,\dots, a_k, b_1,\dots,b_k, a_{k+1}=b_{k+1}, c\}$ be an $n$-element set. Let  $\alpha:=\prt(a_1 b_1)$,
\begin{align*}
\beta&:=\prt(c a_1 b_2; a_2 b_3;\dots;a_{k} b_{k+1})=\prt(a_1 c)\vee  \bigvee_{i\in[k]}\prt(a_i b_{i+1}),  \cr
\gamma&:=\prt(c b_1 a_2; b_2 a_3;\dots;b_{k} a_{k+1})=\prt(b_1 c)\vee  \bigvee_{i\in[k]}\prt(b_i a_{i+1}), 
\end{align*}
and $\delta :=\prt(a_1\dots a_k; b_1\dots b_k)$. Then $\Phi:=\set{\alpha,\beta,\gamma,\delta}$ generates $\Part B$.
\end{lemma}

\stskip

\begin{proof} 
For $k=8$, the situation is visualized by the second graph in Figure \ref{figa}.
Let $S$ be the sublattice generated by $\Phi$ in $\Part B$. 
As $A$ from Lemma \ref{lemma:oddat} is  subset of $B$, 
the definition of  $\beta_0$ and $\gamma_0$ in \eqref{eq:loddata}--\eqref{eq:loddatb}
makes sense.  We still have that $\beta_0=\beta\wedge(\alpha\vee\delta)\in S$ and 
$\gamma_0=\gamma\wedge(\alpha\vee\delta)\in S$. So the sublattice $S_0$ generated by $\Phi_0:=\set{\alpha,\beta_0,\gamma_0,\delta}$ is included in $S$. 
The argument \eqref{align:odda}--\eqref{eq:mWghpstx} needs no change to yield the validity of \eqref{eq:mndL9gra} for $B\setminus\set c$.
That is, if $(x,y)$ is an edge of the graph and $\set{x,y}\cap\set{a_{k+1}=b_{k+1},c}=\emptyset$, then $\prt(x y)\in S_0\subseteq S$. 
As $c$ and $a_{k+1}=b_{k+1}$ are geometrically far enough from each other, the equalities in \eqref{eq:VjrW6gh}  remain valid. (Note that they would fail for  $k=2$.) Hence,  we obtain that $\prt(a_k a_{k+1}), \prt(b_k b_{k+1})\in S$.  Applying Lemma \ref{lemma:circle} to these two memberships,
$\prt(a_1 c)=\beta\wedge\bigl(\prt(a_1 b_1)\vee \gamma \bigr)\in S$, 
$\prt(c b_1)=\gamma\wedge\bigl(\beta\vee \prt(a_1 b_1)\bigr)\in S$, and  \eqref{eq:mndL9gra},  we conclude the validity of Lemma \ref{lemma:evenat}.
\end{proof}
 
\stskip

\begin{lemma}\label{lemma:evenhtwo}
For $2\leq k\in\Nplu$ and $n=2k+2$, let $C=\{a_1,\dots, a_k, b_1,\dots,b_k, a_{k+1}=b_{k+1}, c\}$ be an $n$-element set. Let  $\alpha:=\prt(a_1 b_1 c)$,
\begin{align*}
\beta&:=\prt(a_1 b_2;a_2 b_3;\dots;a_{k} b_{k+1})=\bigvee_{i\in[k]}\prt(a_i b_{i+1}),  \cr
\gamma&:=\prt(b_1 a_2;b_2 a_3;\dots;b_{k} a_{k+1})= \bigvee_{i\in[k]}\prt(b_i a_{i+1}), 
\end{align*}
and $\delta :=\prt(a_1\dots a_k; b_1\dots b_k; ca_{k+1})$. Then $\Phi:=\set{\alpha,\beta,\gamma,\delta}$ generates $\Part C$.
\end{lemma}

\stskip

\begin{proof} For $k=8$, the situation is visualized by the third graph in Figure \ref{figa}. Let $S$ be the sublattice generated by $\set{\alpha,\beta,\gamma,\delta}$ in $\Part C$.
Since $A$   from Lemma \ref{lemma:oddat} is subset of $C$,  there is a natural embedding $f\colon \Part A \to \Part C$; for $\kappa\in\Part A$,  we obtain $f(\kappa)$ from $\kappa$ by adding a singleton block $\set x$ to it for all $x\in C\setminus A$.  (Now there is only one such $x$, namely, $x=c$.)
To distinguish the partitions of $A$ from the members of $\Part C$,  we add $A$ as a subscript to each of the partitions defined in Lemma \ref{lemma:oddat}, and we let $\Phi_A:=\set{\alpha_A,\beta_A,\gamma_A,\delta_A}$. The partition  
$\mu:=\beta\vee \gamma\in \Part C$ has only two blocks, the singleton $\set c$ and $A$. Hence, for every $\kappa\in \Phi$, $f(\kappa_A)=\kappa\wedge \mu\in S$, that is, 
$f(\alpha_A)=\alpha\wedge \mu\in S$, \dots, $f(\delta_A)=\delta\wedge \mu\in S$. 
Since $f$ is an embedding and, by Lemma \ref{lemma:oddat}, $\Phi_A$ generates
$\Part A$, we obtain that $f(\Part A)\subseteq S$. 
In particular, if $x,y\in A$, $x\neq y$, and $c\notin\set{x,y}$, then $\prt(x y)\in S$. Hence, $\prt(a_1 c)=\alpha\wedge\bigl(\prt(a_1 b_{k+1}) \vee\delta\bigr)\in S$
and $\prt(c b_1)=\alpha\wedge\bigl(\delta\vee \prt(b_{k+1}  b_1)\bigr)\in S$. Finally, applying Lemma \ref{lemma:circle} to the \quotat{perimeter} of the graph, we conclude Lemma \ref{lemma:evenhtwo}.
\end{proof}

\stskip

\begin{lemma}\label{lemma:oddhtwo}
For $3\leq k\in\Nplu$ and $n=2k+3$, let $D=\{a_1,\dots, a_k, b_1,\dots,b_k, a_{k+1}=b_{k+1}, c, d\}$ be an $n$-element set. Let  $\alpha:=\prt(a_1 b_1 c)$,
\begin{align*}
\beta&:=\prt(d a_1 b_2; a_2 b_3; \dots; a_{k} b_{k+1})=\prt(d a_1)\vee  \bigvee_{i\in[k]}\prt(a_i b_{i+1}),  \cr
\gamma&:=\prt(d b_1 a_2; b_2 a_3;\dots;b_{k} a_{k+1})= \prt(d b_1)\vee  \bigvee_{i\in[k]}\prt(b_i a_{i+1}), 
\end{align*}
and $\delta :=\prt(a_1\dots a_k; b_1\dots b_k; c a_{k+1})$. Then $\Phi:=\set{\alpha,\beta,\gamma,\delta}$ generates $\Part D$.
\end{lemma}

\stskip

\begin{proof} We apply the same technique as in the previous proof, but now we derive the statement from  Lemma \ref{lemma:evenhtwo}. The objects defined in Lemma \ref{lemma:evenhtwo} will be subscripted by $C$, and the natural embedding $g\colon\Part C\to \Part D$ is defined analogously to $f$ in the previous proof. The sublattice generated by $\Phi$ is denoted by $S$.
With $\nu:=\alpha\vee \delta$, we have that $g(\kappa_C)=\kappa\wedge \nu\in S$ for every $\kappa\in \Phi$. Since $\set{\kappa_C:\kappa\in \Phi}$ generates $\Part C$ by Lemma \ref{lemma:evenhtwo}, we obtain that $g(\Part C)\subseteq S$. Hence, for any  $x,y\in D\setminus\set d=C$ such that $x\neq y$,  $\prt(x y)$ belongs to $S$. Thus,
$\prt(a_1 d)=\beta\wedge\bigl(\prt(a_1 b_1)\vee\gamma \bigr)\in S$ and 
$\prt(d b_1)=\gamma\wedge\bigl(\beta\vee \prt(a_1 b_1)\bigr)$. Therefore, we can apply Lemma \ref{lemma:circle} to conclude Lemma \ref{lemma:oddhtwo} so that, say, we take the pairs $(b_1,c)$ and $(c,b_2)$ instead of the edge $(b_1,b_2)$ on the perimeter.  
\end{proof}

\stskip

\begin{lemma}(Z\'adori \cite{zadori}) \label{lemma:ZLgen}
For $5\leq n\in\Nplu$, $\prt(\set{1,2}; \set{3,4})$ belongs to a four-element generating set of $\Part n$.
\end{lemma}

\stskip

Note that \emph{explicitly}, \cite{zadori} contains this statement only for $n\geq 7$; see the last but one paragraph on page 583 and the last paragraph above the bibliographic section on page 585 in \cite{zadori}. 
However, this should not pose any trouble, as each of the following three independent reasons is sufficient in itself. 
First, Z\'adori's construction remains valid for all $n\geq 5$ if we modify his  $U_4$ to a smaller partition of type $2+2$.  Second, this modification occurs in  (2.8) and Lemma 2.3 in \cite{czgDEBRauth}. Third, in the present paper, 
\eqref{n5ath23} in Lemma \ref{lemma:n5ath2} and 
\eqref{n6atm2} in Lemma \ref{lemma:n6atm} take care of $n\in\set{5,6}$.

\stskip

Next, we present four lemmas to settle some sporadic cases.

\begin{lemma}\label{lemma:nis4}
Each of the sets $\Phi_4:=\set{\prt(12),\prt(23),\prt(34),\prt(41)}$ and $\Psi_4=\{
\prt(12;34)$, $\prt(23)$, $\prt(124)$, $\prt(134)\}$ generates $\Part 4$. 
\end{lemma}

\stskip

\begin{proof} For $\Phi_4$, Lemma \ref{lemma:circle} immediately applies. Let $S$ be the sublattice generated by $\Psi_4$. Then $\prt(12)=\prt(12;34)\wedge\prt(124)\in S$, $\prt(34)=\prt(12;34)\wedge \prt(134)\in S$, and $\prt(41)=\prt(124)\wedge\prt(134)\in S$ show that $\Phi_4\subseteq S$. Hence, $S=\Part 4$, as required.
\end{proof}

\stskip

\begin{lemma}\label{lemma:n5ath2} With
\allowdisplaybreaks{%
\begin{align}
 \alpha&=\prt(123)  , \label{n5ath21} \\
 \beta&=\prt(35)  , \label{n5ath22} \\
 \gamma&=\prt(25;34)  , \text{ and }\label{n5ath23} \\
 \delta&=\prt(145)  , \label{n5ath24}
\end{align} 
}
$\Phi_5:=\set{\alpha,\beta,\gamma,\delta}$ generates $\Part 5$.
\end{lemma}

\stskip

\begin{proof} Let $S$ denote the sublattice generated by $\Phi$. The following elements belong to $S$:
\allowdisplaybreaks{%
\begin{align}
\prt(1235)&=\prt(123)\vee \prt(35)  \text{ by \eqref{n5ath21} and \eqref{n5ath22}}, \label{n5ath25} \\
\prt(2345)&=\prt(35)\vee \prt(25;34)  \text{ by \eqref{n5ath22} and \eqref{n5ath23}}, \label{n5ath28} \\
\prt(1345)&=\prt(35)\vee \prt(145)  \text{ by \eqref{n5ath22} and \eqref{n5ath24}}, \label{n5ath29} \\
\prt(23)&=\prt(123)\wedge \prt(2345)  \text{ by \eqref{n5ath21} and \eqref{n5ath28}}, \label{n5ath210} \\
\prt(34)&=\prt(25;34)\wedge \prt(1345)  \text{ by \eqref{n5ath23} and \eqref{n5ath29}}, \label{n5ath213} \\
\prt(15)&=\prt(145)\wedge \prt(1235)  \text{ by \eqref{n5ath24} and \eqref{n5ath25}}, \label{n5ath214} \\
\prt(45)&=\prt(145)\wedge \prt(2345)  \text{ by \eqref{n5ath24} and \eqref{n5ath28}}, \label{n5ath215} \\
\prt(125;34)&=\prt(25;34)\vee \prt(15)  \text{ by \eqref{n5ath23} and \eqref{n5ath214}}, \text{ and}\label{n5ath222} \\
\prt(12)&=\prt(123)\wedge \prt(125;34)  \text{ by \eqref{n5ath21} and \eqref{n5ath222}}. \label{n5ath234}
\end{align}
}%
Hence, $\prt(12)\in S$ by \eqref{n5ath234}, $\prt(23)\in S$ by \eqref{n5ath210}, $\prt(34)\in S$ by \eqref{n5ath213}, $\prt(45)\in S$ by \eqref{n5ath215}, and $\prt(51)\in S$ by \eqref{n5ath214}. Thus, $\Phi_5$ generates $\Part 5$ by Lemma \ref{lemma:circle}.
\end{proof}

\stskip

The author has developed a computer program package called equ2024p, which is available on his website  \href{https://www.math.u-szeged.hu/~czedli/}{https://www.math.u-szeged.hu/$\sim$czedli/} \ = \ \href{https://tinyurl.com/g-czedli/}{https://tinyurl.com/g-czedli/}.  Instead of the proof above, one can use this package to compute the sublattice generated by $\Phi_5$. 
However, providing a rigorous proof that the program package operates correctly would be extremely difficult, essentially more difficult than verifying the entire paper (including the Appendix)  with all its proofs. 
Therefore, we have elaborated some humanly readable proofs like the one above. As the reader would hardly find more such proofs worth reading, the proofs of the following two lemmas go to the (Appendix) Section \ref{sect:appendix}; they are longer than the proof above.

\begin{lemma}\label{lemma:n6atm} With
\allowdisplaybreaks{%
\begin{align}
 \alpha&=\prt(12)  , \label{n6atm1} \\
 \beta&=\prt(25;34)  , \label{n6atm2} \\
 \gamma&=\prt(13;56)  , \text{ and}\label{n6atm3} \\
 \delta&=\prt(24;36)  , \label{n6atm4}
 \end{align}
 }
$\Phi_6:=\set{\alpha,\beta,\gamma,\delta}$ generates $\Part 6$.
\end{lemma}

\stskip

\begin{lemma}\label{lemma:n7h2-b}  With 
\allowdisplaybreaks{%
\begin{align}
 \alpha&=\prt(123)  , \label{n7h2-b1} \\
 \beta&=\prt(147;56)  , \label{n7h2-b2} \\
 \gamma&=\prt(357;46)  , \text{ and} \label{n7h2-b3} \\
 \delta&=\prt(15;26;34)  , \label{n7h2-b4}
 \end{align}
 }
 $\Phi_7:=\set{\alpha,\beta,\gamma,\delta}$ generates $\Part 7$.
\end{lemma}

\stskip

Now we are ready to prove the main result, Theorem \ref{thm:main}

\stskip

\begin{proof}[Proof of Theorem \ref{thm:main}] Assume that $4\leq n\in\Nplu$, and let $a,b,c,d$ be pairwise distinct elements of an $n$-element set $A$. 
If $\alpha\in \Part A$ is of height 1 or 2, then there is a permutation $\pi$ of $A$ such the automorphism of $\Part A$ determined by $\pi$ sends $\alpha$ to one of the following three partitions: $\alpha_1:=\prt(a b)$, $\alpha_2:=\prt(a b;c d)$, and $\alpha_3:=\prt(a b c)$. 
It suffices to show that each of $\alpha_1$, $\alpha_2$ and $\alpha_3$ belongs to a four-element generating set of $\Part A$ for at least one choice of $(a,b,c,d)$ and the $n$-element set $A$.
The case of  $\alpha_1$ is settled by  Lemma \ref{lemma:oddat} (for $n=5,7,9,11\dots$) , Lemma \ref{lemma:evenat} (for $n=8,10,12,\dots$), 
Lemma \ref{lemma:nis4} (for $n=4$), and Lemma \ref{lemma:n6atm} (for $n=6$). 
Lemmas \ref{lemma:ZLgen} and \ref{lemma:nis4} take care of $\alpha_2$. 
Finally, for $\alpha_3$, we can apply Lemma \ref{lemma:evenhtwo} (for $n=6,8,10,12,\dots$), 
Lemma \ref{lemma:oddhtwo} (for $n=9,11,13\dots$), Lemma \ref{lemma:nis4} (for $n=4$), Lemma \ref{lemma:n5ath2} (for $n=5$), and
Lemma \ref{lemma:n7h2-b} (for $n=7$). The proof of Theorem \ref{thm:main} is complete.
\end{proof}

\stskip
The proof of  Proposition \ref{prop:aleph0} runs as follows.

\stskip

\begin{proof}[Proof of Proposition \ref{prop:aleph0}]  No matter which set of size $\aleph_0$ and which atom in $\Part A$ are taken. Hence, we let $A:=\set{a_i:i\in \Nplu}\cup\set{b_i:i\in\Nplu}$, and define
$\alpha:=\prt(a_1 b_1)$. Let 
$\beta:=\bigvee\set{\prt(a_i b_{i+1}):i\in\Nplu}$,  
$\gamma:=\bigvee\set{\prt(b_i a_{i+1}):i\in\Nplu}$, and $\delta:=\bigvee\set{\prt(a_i a_{i+1};b_i b_{i+1}):i\in\Nplu}$.  Denote $\set{\alpha,\beta,\gamma,\delta}$ by $\Phi$. 
To visualize the construct, remove $a_9=b_9$  and its two adjacent edges from the graph drawn for Lemma \ref{lemma:oddat} on the left side of Figure \ref{figa}.
Then add \quotat{$\vdots$} (three dots) above $a_8$ and $b_8$ to indicate the continuation upwards. 
By \emph{the graph} in the present proof, we mean what we obtain in this way even though we do not draw it. Fortunately, the proof of  Lemma \ref{lemma:oddat} works with almost no changes. Now $\beta_0:=\beta\wedge(\alpha\vee \delta)$ equals $\beta$ and $\gamma_0:=\gamma\wedge(\alpha\vee \delta)$ equals $\gamma$. 
Hence, $[\Phi]$ in Proposition \ref{prop:aleph0} corresponds to $S_0$ in Lemma \ref{lemma:oddat}. 
 Clearly, \eqref{align:odda}, \dots, \eqref{align:oddtket}, and (for all $i$) \eqref{eq:mWghpstx} need no change to imply \eqref{eq:mndL9gra}.
Combining \eqref{eq:mndL9gra} with Lemma \ref{lemma:circle}, it follows that $[\Phi]$ contains all atoms of $\Part A$.  Thus, as each element of $\Pfin A$ is the join of finitely many atoms of $\Part A$, we have that $\Pfin A\subseteq [\Phi]$, as required. 
Furthermore, since each element of $\Part A$ is the join of all (not necessarily finitely many) atoms below it, $[\Phi]$  generates $\Part A$ as a complete lattice. Hence, so does $\Phi$, completing the proof of Proposition \ref{prop:aleph0}.
\end{proof}

\stskip

Corollary \ref{cor:nm} follows easily from an argument given in \cite{czgirk} rather than from an easy-to-quote statement from \cite{czgirk}. To maintain brevity, we do not reproduce this argument. Instead, we present the 7-tuple to which the argument in \cite{czgirk} applies, and highlight which part of \cite{czgirk} is relevant.

\stskip

\begin{proof}[Proof of Corollary \ref{cor:nm}]
Instead of $[n]$, it suffices to take the set $A$ defined in Lemma \ref{lemma:oddat}; see also the graph on the left of Figure \ref{figa}. Let $k:=(n-1)/2$, that is, $n=2k+1$.
Difference up to automorphism does not count, whereby we can assume that $\alpha=\prt(a_1 b_1)$. Define $\beta$, $\gamma$, and $\delta$ in the same way as in  Lemma \ref{lemma:oddat}.  Denote the smallest and the largest partition of $A$ by $\botA$ and $\topA$, respectively. 
With $(u,v):=(a_k,a_{k+1})$, we have that $\beta\vee\gamma=\topA$, 
$\beta\wedge\gamma=\botA$, $\alpha\vee\delta\vee \prt(u v)=\topA$, 
$\alpha\wedge\bigl(\delta\vee \prt(u v)\bigr)=\botA$, and
$\delta\wedge\bigl(\alpha\vee \prt(u v)\bigr)=\botA$.  
Therefore, $(A;\beta,\gamma,\alpha,\delta; u,v)$ is an eligible system in the sense of 
\cite{czgirk}.  
The second part of the proof of the main theorem in \cite{czgirk} begins with a $9$-element eligible system $\mathfrak A_0$. (We can disregard the $8$-element system in \cite{czgirk}, since now $n+m$ is odd.) Instead of the $9$-element eligible system $\mathfrak A_0$, now the seven-tuple above is an $n$-element one. If we change 9 to $n$ and $\mathfrak A_0$  to the $n$-element 7-tuple defined above, then the proof in \cite{czgirk} needs no further change to yield at least 
$2^{m-3}\cdot(m-1)!/(3m+3)$ many $\Phi'$, as required. Thus, Corollary \ref{cor:nm} holds.
\end{proof}

\onlyext{

\section{Appendix}\label{sect:appendix}

\begin{proof}[Proof of Lemma \ref{lemma:n6atm}]
 Let $S$ denote the sublattice generated by $\Phi_6$.  Then  $S$ 
contains
\allowdisplaybreaks{%
\begin{align}
\prt(125;34)&=\prt(12)\vee \prt(25;34)  \text{ by \eqref{n6atm1} and \eqref{n6atm2}}, \label{n6atm5} \\
\prt(124;36)&=\prt(12)\vee \prt(24;36)  \text{ by \eqref{n6atm1} and \eqref{n6atm4}}, \label{n6atm8} \\
\prt(134;256)&=\prt(25;34)\vee \prt(13;56)  \text{ by \eqref{n6atm2} and \eqref{n6atm3}}, \label{n6atm9} \\
\prt(23456)&=\prt(25;34)\vee \prt(24;36)  \text{ by \eqref{n6atm2} and \eqref{n6atm4}}, \label{n6atm10} \\
\prt(1356;24)&=\prt(13;56)\vee \prt(24;36)  \text{ by \eqref{n6atm3} and \eqref{n6atm4}}, \label{n6atm11} \\
\prt(56)&=\prt(13;56)\wedge \prt(23456)  \text{ by \eqref{n6atm3} and \eqref{n6atm10}}, \label{n6atm13} \\
\prt(15)&=\prt(125;34)\wedge \prt(1356;24)  \text{ by \eqref{n6atm5} and \eqref{n6atm11}}, \label{n6atm14} \\
\prt(14)&=\prt(124;36)\wedge \prt(134;256)  \text{ by \eqref{n6atm8} and \eqref{n6atm9}}, \label{n6atm16} \\
\prt(124)&=\prt(12)\vee \prt(14)  \text{ by \eqref{n6atm1} and \eqref{n6atm16}}, \label{n6atm21} \\
\prt(1356)&=\prt(13;56)\vee \prt(15)  \text{ by \eqref{n6atm3} and \eqref{n6atm14}}, \label{n6atm25} \\
\prt(134;56)&=\prt(13;56)\vee \prt(14)  \text{ by \eqref{n6atm3} and \eqref{n6atm16}}, \label{n6atm26} \\
\prt(34)&=\prt(25;34)\wedge \prt(134;56)  \text{ by \eqref{n6atm2} and \eqref{n6atm26}}, \label{n6atm41} \\
\prt(24)&=\prt(24;36)\wedge \prt(124)  \text{ by \eqref{n6atm4} and \eqref{n6atm21}}, \text{ and} \label{n6atm44} \\
\prt(36)&=\prt(24;36)\wedge \prt(1356)  \text{ by \eqref{n6atm4} and \eqref{n6atm25}}. \label{n6atm45}
\end{align}
}%
Since  $\prt(12)\in S$ by \eqref{n6atm1}, $\prt(24)\in S$ by \eqref{n6atm44}, $\prt(43)\in S$ by \eqref{n6atm41}, $\prt(36)\in S$ by \eqref{n6atm45}, $\prt(65)\in S$ by \eqref{n6atm13}, and $\prt(51)\in S$ by \eqref{n6atm14},  $\Phi_6$ generates $\Part 6$, as  required.
\end{proof}

\stskip

\begin{proof}[Proof of Lemma \ref{lemma:n7h2-b}]  
The sublattice $S$ generated by $\Phi_7$ contains
\allowdisplaybreaks{%
\begin{align}
\prt(12347;56)&=\prt(123)\vee \prt(147;56)  \text{ by \eqref{n7h2-b1} and \eqref{n7h2-b2}}, \label{n7h2-b6} \\
\prt(12357;46)&=\prt(123)\vee \prt(357;46)  \text{ by \eqref{n7h2-b1} and \eqref{n7h2-b3}}, \label{n7h2-b7} \\
\prt(123456)&=\prt(123)\vee \prt(15;26;34)  \text{ by \eqref{n7h2-b1} and \eqref{n7h2-b4}}, \label{n7h2-b8} \\
\prt(17)&=\prt(147;56)\wedge \prt(12357;46)  \text{ by \eqref{n7h2-b2} and \eqref{n7h2-b7}}, \label{n7h2-b12} \\
\prt(14;56)&=\prt(147;56)\wedge \prt(123456)  \text{ by \eqref{n7h2-b2} and \eqref{n7h2-b8}}, \label{n7h2-b13} \\
\prt(37)&=\prt(357;46)\wedge \prt(12347;56)  \text{ by \eqref{n7h2-b3} and \eqref{n7h2-b6}}, \label{n7h2-b14} \\
\prt(35;46)&=\prt(357;46)\wedge \prt(123456)  \text{ by \eqref{n7h2-b3} and \eqref{n7h2-b8}}, \label{n7h2-b15} \\
\prt(34)&=\prt(15;26;34)\wedge \prt(12347;56)  \text{ by \eqref{n7h2-b4} and \eqref{n7h2-b6}}, \label{n7h2-b16} \\
\prt(15)&=\prt(15;26;34)\wedge \prt(12357;46)  \text{ by \eqref{n7h2-b4} and \eqref{n7h2-b7}}, \label{n7h2-b17} \\
\prt(1234)&=\prt(123)\vee \prt(34)  \text{ by \eqref{n7h2-b1} and \eqref{n7h2-b16}}, \label{n7h2-b25} \\
\prt(14567)&=\prt(147;56)\vee \prt(15)  \text{ by \eqref{n7h2-b2} and \eqref{n7h2-b17}}, \label{n7h2-b28} \\
\prt(34567)&=\prt(357;46)\vee \prt(34)  \text{ by \eqref{n7h2-b3} and \eqref{n7h2-b16}}, \label{n7h2-b29} \\
\prt(157;26;34)&=\prt(15;26;34)\vee \prt(17)  \text{ by \eqref{n7h2-b4} and \eqref{n7h2-b12}}, \label{n7h2-b31} \\
\prt(1456)&=\prt(14;56)\vee \prt(15)  \text{ by \eqref{n7h2-b13} and \eqref{n7h2-b17}}, \label{n7h2-b44} \\
\prt(3456)&=\prt(35;46)\vee \prt(34)  \text{ by \eqref{n7h2-b15} and \eqref{n7h2-b16}}, \label{n7h2-b48} \\
\prt(14)&=\prt(147;56)\wedge \prt(1234)  \text{ by \eqref{n7h2-b2} and \eqref{n7h2-b25}}, \label{n7h2-b52} \\
\prt(47;56)&=\prt(147;56)\wedge \prt(34567)  \text{ by \eqref{n7h2-b2} and \eqref{n7h2-b29}}, \label{n7h2-b53} \\
\prt(56)&=\prt(147;56)\wedge \prt(3456)  \text{ by \eqref{n7h2-b2} and \eqref{n7h2-b48}}, \label{n7h2-b55} \\
\prt(46;57)&=\prt(357;46)\wedge \prt(14567)  \text{ by \eqref{n7h2-b3} and \eqref{n7h2-b28}}, \label{n7h2-b58} \\
\prt(57)&=\prt(357;46)\wedge \prt(157;26;34)  \text{ by \eqref{n7h2-b3} and \eqref{n7h2-b31}}, \label{n7h2-b59} \\
\prt(46)&=\prt(357;46)\wedge \prt(1456)  \text{ by \eqref{n7h2-b3} and \eqref{n7h2-b44}}, \label{n7h2-b60} \\
\prt(1256;347)&=\prt(15;26;34)\vee \prt(47;56)  \text{ by \eqref{n7h2-b4} and \eqref{n7h2-b53}}, \label{n7h2-b100} \\
\prt(157;2346)&=\prt(15;26;34)\vee \prt(46;57)  \text{ by \eqref{n7h2-b4} and \eqref{n7h2-b58}}, \label{n7h2-b102} \\
\prt(12)&=\prt(123)\wedge \prt(1256;347)  \text{ by \eqref{n7h2-b1} and \eqref{n7h2-b100}}, \text{ and}\label{n7h2-b183} \\
\prt(23)&=\prt(123)\wedge \prt(157;2346)  \text{ by \eqref{n7h2-b1} and \eqref{n7h2-b102}}. \label{n7h2-b184}
\end{align}
}%
So  $\prt(12)\in S$ by \eqref{n7h2-b183}, $\prt(23)\in S$ by \eqref{n7h2-b184}, $\prt(37)\in S$ by \eqref{n7h2-b14}, $\prt(75)\in S$ by \eqref{n7h2-b59}, $\prt(56)\in S$ by \eqref{n7h2-b55}, $\prt(64)\in S$ by \eqref{n7h2-b60}, and $\prt(41)\in S$ by \eqref{n7h2-b52}. Thus,  Lemma \ref{lemma:circle} concludes the proof of 
Lemma \ref{lemma:n7h2-b} follows from Lemma \ref{lemma:circle}.
 \end{proof}

\stskip

} 

\section*{Acknowledgements}
This research was supported by the National Research, Development and Innovation
Fund of Hungary, under funding scheme K 138892.

\bibliographystyle{plain}
\bibliography{czgh2pgen}

\end{document}